\RequirePackage{fix-cm}
\documentclass[smallcondensed]{svjour3}       
\smartqed  
\usepackage{graphicx}

\usepackage{amssymb}
\usepackage{amsmath}
\newtheorem{thm}{Theorem}
\newtheorem{lem}[thm]{Lemma}

\newtheorem{defi}{Definition}

\begin{document}

\title{A note on algebraic potentials and Morales-Ramis theory
}

\titlerunning{A note on algebraic potentials and Morales-Ramis theory}        

\author{Thierry Combot}


\institute{T. Combot \at
              IMCCE 77 Avenue Denfert Rochereau 75014 PARIS\\
              \email{combot@imcce.fr} 
}

\date{Received: date / Accepted: date}

\maketitle

\begin{abstract}
We present various properties of algebraic potentials, and then prove that some Morales-Ramis theorems readily apply for such potentials even if they are not in general meromorphic potentials. This allows in particular to precise some non-integrability proofs in celestial mechanics, where the mutual distances between the bodies appear in the potentials, and thus making this analysis unavoidable.
\end{abstract}

\section{Introduction}\label{intro}

The purpose of this note is to apply the following Theorem for proving ``meromorphic'' non-integrability of algebraic potentials

\begin{thm}\label{thm:Morales0} (\cite{23} Theorem 2.) Let us consider a symplectic analytical complex manifold $M$ of dimension $2n$, with the Poisson bracket defined by the symplectic form, $H$ a Hamiltonian analytic on $M$ and $\Gamma \subset M$ a particular (not a point) orbit. If $H$ possesses a complete system of first integrals in involution, functionally independent and meromorphic on a neigbourhood of $\Gamma$, then the identity component of the Galois group of variational equations is abelian at any order.
\end{thm}

This Theorem is an extension of Theorem 7 of \cite{5}. There is an application for homogeneous potentials in \cite{22}. In particular, this last Theorem is directly applied to celestial mechanics problems like in \cite{8}, \cite{9}, \cite{37}, \cite{31} where the potential is algebraic. Still, the Theorem requires explicitly that the potential should be meromorphic. Clearly an algebraic potential on $\mathbb{C}^n$ cannot be meromorphic on $\mathbb{C}^n$ unless it is rational. The problem is even worse than just being not regular enough because such potential is in fact multivalued (and thus is not a function, even on a small open set). In many other articles, the problem is either ignored or not suitably analyzed, as in \cite{9}, \cite{32}, \cite{25}. This makes these proofs ambiguous as both the dynamical system and the notion of meromorphic integrability are not well defined.

Still Theorem \ref{thm:Morales0} could be used if we consider such algebraic potential as a rational function on an algebraic complex manifold $M$ instead of $\mathbb{C}^n$. The aim of this note is to give a clear definition of an algebraic potential, the associated dynamical system and what is a ``meromorphic first integral'' in this case. This will give some precisions on several non-integrability proofs in celestial mechanics, in particular Theorem 10 of \cite{45}. A typical example is the following
$$V(q_1,q_2)=( q_1^2+q_2^2)^{3/2}$$
On $\mathbb{C}^2$, this expression is not meromorphic, as it is not even single-valued. The main idea to circumvent such a problem is to introduce algebraic extensions
\begin{equation}\label{example}
V(q_1,q_2,w_1)=w_1^3 \qquad w_1^2-q_1^2-q_2^2=0
\end{equation}
and then to see the function $V$ as a function well defined on the $2$-dimensional algebraic variety $\{(q_1,q_2,w_1)\in\mathbb{C}^3,\;\; w_1^2-q_1^2-q_2^2=0\}$. Let us now present more general statements.

\bigskip

We consider polynomials $G_1,\dots,G_s\in\mathbb{C}[q_1,\dots,q_n,w_1,\dots,w_s]$ and the ideal $I=<\! G_1,\dots,G_s\!>$. In the following, we will assume that $I$ is a prime ideal and that the matrix
$$J\in M_s(\mathbb{C}[q,w]) \qquad J_{i,j}=\frac{\partial G_i}{\partial w_j},\quad i,j=1\dots s$$
has a non-zero determinant modulo the ideal $I$. We define the associated manifold $\mathcal{S}=I^{-1}(0)$ and $\pi:\mathcal{S} \longrightarrow \mathbb{C}^n$ the projection on variables $q$.

\bigskip

A holomorphic function on a non empty open set $U\subset \mathcal{S}$ is by definition, locally the restriction of holomorphic functions on open sets $W\subset \mathbb{C}^{n+s}$ to $U$. A meromorphic function on $U$ is locally a quotient $h/k$ of two holomorphic functions $h,k$, with $k$ non identically zero.


Let us now define derivations on $\mathcal{S}$. We first introduce the set
$$\Sigma(I)=\{(q,w)\in\mathcal{S},\;\; \hbox{det}(J)(q,w)=0\}$$
This set will be called the critical set and corresponds to points on $S$ where the Jacobian matrix $J$ of the application $w\longrightarrow (G_1,\dots,G_s)$ is not invertible. In example \eqref{example}, we have in particular $\Sigma(I)=\{w_1=0,\;q_1\pm i q_2=0\}$. Remark that this set is at least of codimension one because the determinant is not zero modulo $I$. The manifold $\mathcal{S}$ is of dimension $n$, as it is the common zero of $s$ functionally independent ($\hbox{det}(J)\neq 0$) polynomials in dimension $n+s$.

Let $U$ be a non empty open set of $\mathcal{S}$ and $f$ a meromorphic function on $U$. We may now define
\begin{align}\label{deriv}
\frac{\partial f}{\partial q_k}= \partial_k f-
\left( \partial_{n+1} f,\dots,\partial_{n+s} f\right)J^{-1} \left(\partial_k G_1,\dots,\partial_k G_s\right)^\intercal
\end{align}
where $\partial_i$ denotes the derivative according to the $i$-th variable (the variables are $q_1,\dots,q_n,w_1,\dots,w_s$ in this order). These derivatives are well defined outside $\Sigma(I)$. We define moreover the critical set of $V$
$$\Sigma(V)=\{(q,w)\in U,\;\; V(q,w)\notin \mathbb{C} \} \cup (\Sigma(I)\cap U)$$

\begin{defi}
A meromorphic potential $V$ on an open set $U\subset \mathcal{S}$ defines the following dynamical system on $\mathbb{C}^n\times \left(U \setminus \Sigma(V)\right)$
\begin{equation}\label{eqpot}
\dot{q}_i=p_i,\quad\; \dot{p}_i=-\frac{\partial V}{\partial q_i},\quad i=1\dots n\qquad\quad  \dot{w}_i=\sum\limits_{j=1}^s p_j \frac{\partial w_i}{\partial q_j},\quad i=1\dots s
\end{equation}
\end{defi}

Let us remark now that an algebraic potential fits this definition. Consider an algebraic function $V$ on $\mathbb{C}^n$ and $P\in\mathbb{C}[q_1,\dots,q_n][w_1]$ a non-zero irreducible polynomial such that $P(V(q))=0$. The ideal $I=<\!P\!>$ on $\mathbb{C}[q_1,\dots,q_n,w_1]$ is prime because $P$ is irreducible. The matrix $J$ is $1\times 1$ and its determinant is $\partial_{w_1} P$. As $P$ is a non-zero irreducible polynomial, we have $\partial_w P\neq 0 \hbox{ mod } I$ otherwise $P$ divides $\partial_w P$ which is impossible because $\partial_w P$ is non-zero and of degree lower than $P$. Thus $\mathcal{S}=I^{-1}(0)$ is a manifold of dimension $n$, and $V$ is a meromorphic potential on $\mathcal{S}$, with $V(q,w)=w$. Recall that a meromorphic function on a complex algebraic manifold $\mathcal{S}$ is not strictly speaking a function, as it has singularities, and even indeterminate points, as for example
$$V(x,y)=\frac{xy}{x^2+y^2}$$
which is indeterminate at $(0,0)$, but still is meromorphic on $\mathbb{C}^2$ (and even rational).

\medskip

Given a potential $V$ on $\mathcal{S}$, the corresponding multivalued potential on $\mathbb{C}^n$ is given by $V(\pi^{-1}(q))$. The critical set $\Sigma(I)$ contains all ramification points of the multivalued expression $V(\pi^{-1}(q))$, and the critical set $\Sigma(V)$ contains all ramification/singular/indeterminate points of $V(\pi^{-1}(q))$. Thus defining the derivability in respect to the $q_i$ as in \eqref{deriv}, we find that
$$\Sigma(V)= \{ (q,w)\in \mathcal{S},\;\; V \hbox{ is not } C^\infty \hbox{ at } (q,w)\}$$
We can now apply Theorem \ref{thm:Morales0} to such potentials.

\section{A Morales-Ramis-Simo Theorem for algebraic potentials}

\begin{thm}\label{thmmorales}
Let $V$ be a meromorphic potential on an open set $U\subset \mathcal{S}$ and $\Gamma\subset \mathbb{C}^n\times U$ a non-stationary orbit of $V$. Suppose $\Gamma\not\subset \mathbb{C}^n\times\Sigma(V)$. If there are $n$ first integrals meromorphic on $\mathbb{C}^n\times (U \setminus \Sigma(V))$ of $V$ that are in involution and functionally independent over an open neighbourhood of $\Gamma$, then the identity component of Galois group of the variational equation near $\Gamma$ is abelian over the base field of meromorphic functions on $\Gamma\setminus (\mathbb{C}^n\times\Sigma(V))$.
\end{thm}

\begin{proof}
One just needs to check that hypotheses of Theorem \ref{thm:Morales0} are satisfied. We define $W=\Gamma\cap (\mathbb{C}^n\times\Sigma(V))$. These points $W$ are singularities of the vector field \eqref{eqpot}. Let us remove these points by posing $\Gamma'=\Gamma\setminus W$. Remark that as the curve $\Gamma$ is not contained in $\mathbb{C}^n\times\Sigma(V)$, $\Gamma'$ is still a curve (so a non-stationary orbit). We now consider an open neighborhood $M\subset \mathbb{C}^n\times U$ of $\Gamma'$ such that $M\cap \Sigma(V)=\emptyset$. So $M$ is a complex manifold of dimension $2n$. We endow this manifold with the canonical symplectic structure in $p,q$, where the derivations in $q$ are defined as in equation \eqref{deriv}. This symplectic structure degenerates on $\Sigma(I)$, but we do not care as $M\cap \Sigma(I)=\emptyset$. Knowing that $M\cap \Sigma(V)=\emptyset$, we also know that the corresponding Hamiltonian
$$H(p,q,w)=\frac{1}{2} \sum\limits_{i=1}^n p_i^2 +V(q,w)$$
has no singularities on $M$, and thus is holomorphic. All hypotheses of Theorem \ref{thm:Morales0} are satisfied, and so Theorem \ref{thmmorales} follows.
\end{proof}\qed

So it is possible to readily apply Morales-Ramis Theorem for meromorphic potentials on a complex algebraic manifold. Remark that the additional hypothesis $\Gamma \not\subset \mathbb{C}^n\times\Sigma(V)$ can be important. For example, the potential
\begin{equation}\label{prob}
V(q_1,q_2,w_1)=w_1^5+q_2^2 \qquad I=<w_1^2-q_1>
\end{equation}
has a particular orbit given by $w_1(t)=0,q_1(t)=0,q_2(t)=\cos t$. We have $\Sigma(I)=\{w_1=q_1=0,q_2\in\mathbb{C}\}$. This orbit is non stationary, we could compute the variational equation, but still Theorem \ref{thmmorales} does not apply because it is included in $\Sigma(I)$.

\medskip

Let us now make some precisions about the base field on which we should compute the Galois group. In \cite{5}, it is written that the base field is the field of meromorphic functions on $\Gamma'$, but in all applications, we compute Galois groups over the base field of rational functions. In page 114 of \cite{5}, they do not ignore this difficulty and remark that in case of a Fuchsian variational equation, this will still work because these two Galois groups are equal. However, no explicit proof is given, and so let us prove the following result.

\begin{lem}\label{schle}
Let 
\begin{equation}\label{eqex}
\dot{x}=Ax \qquad  A\in M_n(\mathbb{C}(t))
\end{equation}
be a regular singular differential equation (defined in 5.1.2 p 147 of \cite{42}), $D\subset \mathbb{C}$ a discrete set and $K$ the field of meromorphic functions on $\mathbb{C}\setminus D$. The Galois group $G_1$ of equation \eqref{eqex} over the base field $K$ is equal to its Galois group $G_2$ over the base field $\mathbb{C}(t)$.
\end{lem}

\begin{proof}
We consider the resolvant of equation \eqref{eqex} noted $x(t)$. Following Chapter 1.4 of \cite{42}, we define
\begin{equation}\begin{split}\label{inv}
Inv_1=\{P\in K[x_{1,1},\dots,x_{n,n},\left(\hbox{det}\left((x)_{i,j=1\dots n}\right)\right)^{-1}],\;\;P(t,x(t))=0\} \\
Inv_2=\{P\in \mathbb{C}(t)[x_{1,1},\dots,x_{n,n},\left(\hbox{det}\left((x)_{i,j=1\dots n}\right)\right)^{-1}],\;\;P(t,x(t))=0\}
\end{split}\end{equation}
We have then by definition
$$G_1=\{\sigma\in GL_n(\mathbb{C}), \forall P\in Inv_1,\;\;P(t,\sigma x(t))=0)\}$$
$$G_2=\{\sigma\in GL_n(\mathbb{C}), \forall P\in Inv_2,\;\;P(t,\sigma x(t))=0\}$$
We know that $Inv_2 \subset Inv_1$ and so $G_1 \subset G_2$. Let $P\in Inv_1$ and let us consider $\gamma$ a closed curve in $\mathbb{C}\setminus \left( D\cup \{z_1,\dots,z_p\}\right)$ where $z_1,\dots,z_p$ are the singularities of equation \eqref{eqex}, and $\sigma$ the corresponding monodromy element. As the coefficients of $P$ are meromorphic and univalued along the curve $\gamma$, we find that $P(t,\sigma x(t))=0$. Any curve in $\mathbb{C}\setminus \{z_1,\dots,z_p\}$ is homotopic to a curve in $\mathbb{C}\setminus \left( D\cup \{z_1,\dots,z_p\}\right)$, so noting the monodromy group $G_3$, we have $G_3 \subset G_2$.

We now use the Schlesinger density Theorem (\cite{42} Theorem 5.8 p 148). The Galois group $G_2$ is the Zariski closure of the monodromy group $G_3$. So we get
$$G_3\subset G_1 \subset G_2 \qquad\quad \overline{G}_3= G_2$$
As $G_1$ is a Zariski closed group, we finally have $G_1=G_2$.
\end{proof}\qed

Typically, when Theorem \ref{thmmorales} is used in applications, a parametrization $\phi$ of the curve is chosen, and the variational equation is computed according to this parametrization. In most examples, the variational equation obtained is with rational coefficients, regular singular, and the base field $K$ for Galois group computations is an algebraic extension of meromorphic functions on $\mathbb{C}\setminus D$ where $D$ is discrete (and $\phi(D)$ corresponds to singular points of $V$ on $\Gamma$). Then in this case, using Lemma \ref{schle}, the Galois group of the variational equation over $\mathbb{C}(t)$ is a finite extension of the Galois group over the base field $K$, and thus has the same identity component.

\section{Application to homogeneous potentials}

\begin{defi}
Let $V$ be a meromorphic potential on $\mathcal{S}$. We say that $V$ is homogeneous if there exists $(d_1,d_2)\in\mathbb{Z}^*\times\mathbb{Z},\;(k_1,\dots,k_s)\in\mathbb{Z}^s$ such that
\begin{align*}
\forall (q,w)\in \mathcal{S},\alpha\in\mathbb{C}^*,\quad (\alpha^{d_1} q,\alpha^{k_1} w_1,\dots,\alpha^{k_s}w_s)\in\mathcal{S},\\
V(\alpha^{d_1} q,\alpha^{k_1} w_1,\dots,\alpha^{k_s}w_s)=\alpha^{d_2} V(q,w)
\end{align*}
The homogeneity degree of $V$ is then $d_2/d_1$.
\end{defi}

\begin{thm}\label{thmmorales2} (Compare \cite{22})
Let $V$ be a homogeneous meromorphic potential on $\mathcal{S}$ of homogeneity degree $k\in\mathbb{Z}^*$ and $c\in\mathcal{S}\setminus (\{0\}\cup \Sigma(V))$ such that
$$\frac{\partial}{\partial q_i} V(c)= \pi(c)_i\quad i=1\dots n$$
Suppose that $\nabla^2V(c)$ (the Hessian matrix according to derivations in $q$) is diagonalizable. If $V$ has $n$ meromorphic first integrals on $\mathbb{C}^n\times (\mathcal{S}\setminus \Sigma(V))$ which are in involution and functionally independent, then for any $\lambda\in \hbox{Sp}(\nabla^2 V(c))$, the couple $(k,\lambda)$ belongs to the table
\begin{center}\begin{tabular}{|c|c|c|c|}
\hline
$k$&$\lambda$&$k$&$\lambda$\\\hline
$\mathbb{Z}^*$&$\frac{1}{2}\,i \left( ik+k-2 \right)$&$-3$&$\frac{25}{24}-\frac{1}{24}(\frac{6}{5}+6 i)^2$ \\\hline
$\mathbb{Z}^*$&$\frac{1}{2}\left( ik+k-1 \right)  \left( ik+1 \right)/k$&$-3$&$\frac{25}{24}-\frac{1}{24}(\frac{12}{5}+6 i)^2$ \\\hline
$2$&$\mathbb{C}$&$3$&$-\frac{1}{24}+\frac{1}{24}(2+6 i)^2$ \\\hline
$-2$&$\mathbb{C}$&$3$&$-\frac{1}{24}+\frac{1}{24}(\frac{3}{2}+6 i)^2$ \\\hline
$-5$&$\frac{49}{40}-\frac{1}{40}(\frac{10}{3}+10 i)^2$&$3$&$-\frac{1}{24}+\frac{1}{24}(\frac{6}{5}+6 i)^2$ \\\hline
$-5$&$\frac{49}{40}-\frac{1}{40}(4+10 i)^2$&$3$&$-\frac{1}{24}+\frac{1}{24}(\frac{12}{5}+6 i)^2$ \\\hline
$-4$&$\frac{9}{8}-\frac{1}{4}(\frac{4}{3}+4i)^2$&$4$&$-\frac{1}{8}+\frac{1}{8}(\frac{4}{3}+4 i)^2$ \\\hline
$-3$&$\frac{25}{24}-\frac{1}{24}(2+6 i)^2$&$5$&$-\frac{9}{40}+\frac{1}{40}(\frac{10}{3}+10 i)^2$ \\\hline
$-3$&$\frac{25}{24}-\frac{1}{24}(\frac{3}{2}+6 i)^2$&$5$&$-\frac{9}{40}+\frac{1}{40}(4+10 i)^2$ \\\hline
\end{tabular}\end{center}
\end{thm}

\begin{proof}
We want to use Theorem \ref{thmmorales}. As $V$ is homogeneous, there exists $(d_1,d_2)\in\mathbb{Z}^*\times\mathbb{Z},\;(k_1,\dots,k_p)\in\mathbb{Z}^s$ such that
$$V(\alpha^{d_1} q,\alpha^{k_1} w_1,\dots,\alpha^{k_s}w_s)=\alpha^{d_2} V(q,w)$$
We note $k=d_2/d_1$ the homogeneity degree of $V$. We now consider the curve $\Gamma\subset \mathcal{S}$ given by
\begin{align*}
q(t)=\phi(t)^{d_1}.\pi(c),\;\; w(t)=(c_{n+1}\phi(t)^{k_1},\dots,c_{n+s}\phi(t)^{k_s}),\\
p(t)=d_1\dot{\phi}(t)\phi(t)^{d_1-1}.\pi(c) \qquad \textstyle{\frac{1}{2}}d_1^2\dot{\phi}^2\phi^{2d_1-2}=-\frac{d_1}{d_2}\phi^{d_2}+1
\end{align*}
This curve $\Gamma$ is an orbit of $V$. The singular set $\Sigma(V)$ is a homogeneous variety (because $V$ is homogeneous), and as $c\notin \Sigma(V)$, the points of $\Gamma\cap \Sigma(V)$ correspond to $\phi=0$.

The variational equation at first order near the curve $\Gamma$ is a linear differential equation in $X\in\mathbb{C}^{2n+s}$. At each point $(\dot{\phi},\phi)$ of $\Gamma$, the vector $X$ belongs to the tangent space of $\mathbb{C}^n\times\mathcal{S}$. Outside the singular points $\Sigma(V)$, the projection of this tangent space on the $p,q$ variables is $\mathbb{C}^{2n}$. So we can project the variational equation and get a differential equation on $\mathbb{C}^{2n}$. Noting $\nabla^2 V(c)$ the $n\times n$ Hessian matrix in respect to the derivations in $q$, the projected first order variational equation is given by
$$\ddot{X}=-\phi(t)^{d_2-2d_1} \nabla^2 V(c) X$$
We now consider the parametrization of $\Gamma$ by $\phi^{d_2}$ and thus making a variable change $z=\phi(t)^{d_2}$ in the first order variational equation. After diagonalizing the matrix $\nabla^2V(c)$, we obtain $n$ uncoupled hypergeometric equations in $z$
$$z(z-1)\frac{d^2X_i}{dz^2}+\left(\frac{3k-2}{2k}z-\frac{k-1}{k}\right) \frac{dX_i}{dz}-\frac{\lambda_i}{2k} X_i=0 \qquad \lambda_i\in\hbox{Sp}(\nabla^2 V(c))$$
The hypergeometric equation is a Fuchsian equation. The base field on which we should compute the Galois group of this equation is the field $K$ of meromorphic functions in $\phi,\dot{\phi}$ for $\phi\neq 0$, which due to the relation $\textstyle{\frac{1}{2}}d_1^2\dot{\phi}^2\phi^{2d_1-2}=-\textstyle{\frac{d_1}{d_2}}\phi^{d_2}+1$ is an algebraic extension of degree $2d_2$ of the field of meromorphic functions in $\phi^{d_2}$ for $\phi^{d_2}\neq 0$ (the parametrization we have chosen). Using Lemma \ref{schle}, the Galois group $G_2$ of the hypergeometric equation over the base field $K$ has finite index (at most $2d_2$) in the Galois group $G_1$ over the base field $\mathbb{C}(z)$.

So if $G_2$ has an abelian identity component, then it is also the case for $G_1$. The Kimura table \cite{14} gives all the cases where the Galois group $G_1$ of the hypergeometric equation has an abelian identity component, and this produces the table.
\end{proof}\qed

Theorem \ref{thmmorales2} can thus be applied to algebraic potentials, and in particular for the $n$ body problem in dimension $d\geq 2$
$$V=\sum\limits_{1\leq i<j\leq n} \frac{m_im_j}{r_{i,j}} \qquad I=\left\langle\left( r_{i,j}^2-\sum\limits_{k=1}^d (q_{i,k}-q_{j,k})^2 \right)_{1\leq i<j\leq n}\right\rangle$$
where $q_{i,\cdot}$ corresponds to the coordinates of body number $i$. The ideal $I$ is prime for $d\geq 2$ (but not for $d=1$!), and the critical set is
$$\Sigma(V)=\{(q,r)\in \mathcal{S},\;\; \exists i\neq j,\; r_{i,j}=0\}$$
We have moreover that the phenomenon of \eqref{prob} cannot appear. Indeed, all points of $\Sigma(V)$ are singularities of $V$ (and not only ramification points), so we cannot choose a ``bad'' Darboux point (a Darboux point in $\Sigma(V)$).

In the articles \cite{44}, \cite{43}, some generalized problems with other homogeneity degrees are analyzed. For the generalized $3$ body problem in \cite{44}, the authors only consider negative degrees, and so we still have that all points of $\Sigma(V)$ are singularities of $V$. In \cite{43}, such a problem could appear, but they smartly did not use forbidden orbits in their analysis. Thus the non-integrability proofs of \cite{8}, \cite{44}, \cite{9}, \cite{37}, \cite{43}, \cite{31}, \cite{45} are confirmed using the regularity class for first integrals ``meromorphic on $\mathbb{C}^n\times (\mathcal{S}\setminus \Sigma(V))$''.


\label{}
\bibliographystyle{spbasic}
\bibliography{nonzeroangular}

\end{document}